\documentclass[10pt,twoside]{article}
\usepackage{amscd,amsfonts,amsmath,amssymb,amstext,amsthm}
\usepackage{mathpazo}
\usepackage{hyperref}

\def\RR{{\Bbb R}}
\def\CC{{\Bbb C}}

\def\ZZ{{\Bbb Z}}

\def\PP{{\Bbb P}}

\def\ra{\rightarrow}

\newtheorem{theorem}{Theorem}[section]
\newtheorem{lemma}[theorem]{Lemma}
\newtheorem{proposition}[theorem]{Proposition}
\newtheorem{corollary}[theorem]{Corollary}

\theoremstyle{remark}
\newtheorem{remark}[theorem]{Remark}

\newcommand{\hq } {{/\kern -.185em/}}
\newcommand{\he } {{\kern -.050em\sim _H \kern -.050em }}
\newcommand{\eq } {{\kern -.100em\sim \kern -.100em}}
\newcommand{\eqs } {{\kern -.100em\sim }}
  \renewenvironment{thebibliography}[1]{%
    \begin{oldthebibliography}{#1}%
      \setlength{\parskip}{0.3ex}%
      \setlength{\itemsep}{0.3ex}%
  }%
  {%
    \end{oldthebibliography}%
}

\begin{document}
\title{Infinite-dimensionality of the Automorphism Groups of Homogeneous Stein Manifolds}

\author{Alan Huckleberry \& Alexander Isaev}

\date{}

\maketitle

\sloppy
\parindent0ex
\parskip1.2ex
\pagenumbering {arabic}

\thispagestyle{empty}

\pagestyle{myheadings}
\markboth{Alan Huckleberry and Alexander Isaev}{Automorphism Groups of Homogeneous Stein Manifolds}

\begin{abstract} \noindent We show that the group $\hbox{Aut}_{\mathcal O}(X)$ of holomorphic automorphisms of a Stein manifold $X$ with $\hbox{dim}\,X\ge 2$ is infinite-dimensional, provided $X$ is a homogeneous space of a holomorphic action of a complex Lie group.
\end{abstract}

\section{Introduction}
\setcounter{equation}{0}

For a connected complex manifold $X$ let $\hbox{Aut}_{\mathcal O}(X)$ denote the group of holomorphic automorphisms of $X$. The objective of this paper is to show that for a large class of manifolds the group $\hbox{Aut}_{\mathcal O}(X)$ is infinite-dimensional in the sense that the Lie algebra $V_{\mathcal O}^c(X)$ generated by complete holomorphic vector fields on $X$ is infinite-dimensional (we think of a complete holomorphic vector field on $X$ as a complete vector field whose global flow consists of holomorphic transformations of $X$). This condition is equivalent to the impossibility of introducing on $\hbox{Aut}_{\mathcal O}(X)$ a topology with respect to which it becomes a Lie transformation group as defined by Palais (see \cite{P}, p.p. 99, 103). In particular, if $V_{\mathcal O}^c(X)$ is infinite-dimensional, $\hbox{Aut}_{\mathcal O}(X)$ is not a Lie group in the compact-open topology and not even locally compact (see \cite{MZ}, p. 208).  

We note that for large classes of complex manifolds $\hbox{Aut}_{\mathcal O}(X)$ is known to be a Lie group in the compact-open topology. If $X$ is Kobayashi-hyperbolic, then $\hbox{Aut}_{\mathcal O}(X)$ is a real Lie group which is never complex unless it is discrete \cite{Ko}, \cite{Ka}. On the other hand, if $X$ is compact, then $\hbox{Aut}_{\mathcal O}(X)$ is a complex Lie group. In the present paper we consider manifolds homogeneous with respect to holomorphic actions of complex Lie groups. Observe that any such manifold is far from being hyperbolic, since every two points in it can be joined by a finite chain of orbits of complex 1-parameter subgroups. In addition, we assume the manifolds to be Stein, in particular ruling out the compact case. Our main result is the following theorem.

\begin{theorem}\label{main} Let $X$ with $\hbox{dim}\,X\ge 2$ be a connected Stein manifold. If $X$ is a homogeneous space $G/H$ of a holomorphic action of a connected complex Lie group, then $V_{\mathcal O}^c(X)$ is infinite-dimensional.
\end{theorem}

If $X=G$ is a complex algebraic group (here $H$ is trivial), the infinite-dimensionality of $V_{\mathcal O}^c(X)$ follows from Theorem 3 of \cite{KK2} which states that in this case $V_{\mathcal O}^c(X)$ is dense in the infinite-dimensional Lie algebra $V_{\mathcal O}(X)$ of all holomorphic vector fields on $X$. This density property is much stronger than the assertion of Theorem \ref{main}. However, even in the case when Theorem 3 of \cite{KK2} applies, we present new and quite explicit methods for constructing a large number of complete holomorphic vector fields on $X$. Other density results can be found in \cite{KK1}, \cite{TV}, \cite{Varo1}, \cite{Varo2}, \cite{Varo3}. In particular, it follows from Corollary 6.6 of \cite{Varo1} that if $X$ is as in Theorem \ref{main}, then $V_{\mathcal O}^c(X\times\CC)$ is dense in the infinite-dimensional Lie algebra $V_{\mathcal O}(X\times\CC)$.  

Results on the infinite-dimensionality of the automorphism groups of certain complex manifolds homogeneous under actions of real Lie groups can be found in \cite{DZ}. We also note that it was shown in \cite{Wi} that the Stein domain $D:=\{(z,w)\in\CC^2: zw\ne 1\}$ is  $\hbox{Aut}_{\mathcal O}(D)$-homogeneous, but no connected Lie group (real or complex) acts transitively on $D$ by holomorphic transformations (cf. \cite{Ka}, Satz 4.13). Furthermore, the Lie algebra $V_{\mathcal O}^c(D)$ is infinite-dimensional, and thus the converse to Theorem \ref{main} does not hold in general.  

The collection of manifolds covered by Theorem \ref{main} is quite large. For example, every simply-connected complex Lie group $G$ is Stein (indeed, writing the Levi-Malcev decomposition $G=R\ltimes S$, where $R$ is the radical of $G$ and $S$ is the semi-simple Levi subgroup, we observe that $S$ is a linear algebraic group and $R$, being a simply-connected solvable group, is biholomorphic to some $\CC^k$).  Further, if $G$ is a complex reductive group and $H$ is a complex reductive subgroup then $G/H$ is Stein \cite{Mat2}, \cite{Mat3}, \cite{O} (see Remark \ref{algebraic} in Section \ref{methods}). Another class of examples arises when $G$ is nilpotent. In this case $G/H$ is Stein provided it is holomorphically separable \cite{GH} (for other sufficient conditions for $G/H$ to be Stein see \cite{Mat3}). 

Perhaps the best-known example of a Stein manifold with infinite-dimensional automorphism group is $\CC^n$ for $n\ge 2$. Observe that the complex Abelian group $(\CC^n,+)$ acts holomorphically and transitively on $\CC^n$ by translations. A less-known example is given by affine quadrics. Let $Q_n$ be the quadric in the complex projective space $\PP^{n+1}$ and $Q_{(n)}$ the corresponding affine quadric in $\CC^{n+1}$ obtained by deleting from $Q_n$ its intersection with the hyperplane at infinity. While the group of holomorphic automorphisms $\hbox{Aut}_{\mathcal O}(Q_n)$ of $Q_n$ is the complex Lie group $PO_{n+2}(\CC)$ (the quotient of the complex orthogonal group $O_{n+2}(\CC)$ by its center), the automorphism group $\hbox{Aut}_{\mathcal O}(Q_{(n)})$ of $Q_{(n)}$ is infinite-dimensional if $n\ge 2$. Indeed, writing the equation of $Q_{(n)}$ as
$$
z_1z_2+\sum_{j=3}^{n+1}z_j^2=1,
$$
one can easily verify that the following algebraic map lies in $\hbox{Aut}_{\mathcal O}(Q_{(n)})$ for every polynomial $P(z_2)$:

$$
\begin{array}{lll}
z_1&\mapsto& z_1-P(z_2)^2z_2-2P(z_2)z_3,\\
z_2&\mapsto&z_2,\\
z_3&\mapsto& P(z_2)z_2+z_3,\\
z_j&\mapsto& z_j,\quad 3<j\le n+1.
\end{array}
$$   
Such maps form a subgroup of $\hbox{Aut}_{\mathcal O}(Q_{(n)})$ isomorphic to the Abelian group of all polynomials in one complex variable, and thus the Lie algebra $V_{\mathcal O}^c(Q_{(n)})$ is infinite-dimensional. Furthermore, it follows from the results of \cite{KK1} that $V^c_{\mathcal O}(Q_{(n)})$ is dense in $V_{\mathcal O}(Q_{(n)})$. Observe also that $Q_{(n)}$ is a homogeneous space of the linear action of the complex reductive group $SO_{n+1}(\CC)$. It must be stressed here that we treat $Q_{(n)}$ as a complex rather than algebraic manifold and are interested in its full group of holomorphic automorphisms. Results on the finite- and infinite-dimensionality of the groups of algebraic automorphisms of affine quadrics over arbitrary fields can be found in the recent paper \cite{To} and references therein.

The paper is organized in follows. In Section \ref{methods} we outline a general approach to constructing a large number of 1-parameter subgroups of automorphisms for a manifold that is, firstly, endowed with a holomorphic action of a complex Lie group and, secondly, can be fibered over a complex space that admits non-constant holomorphic functions. In Sections \ref{case1} and \ref{case2} these methods are applied to specific fibrations to yield Theorem \ref{main} for non-solvable and solvable groups, respectively.   

{\bf Acknowledgement.} This work was initiated while the first author was visiting the Australian National University.

\section{General Methods}\label{methods}
\setcounter{equation}{0}

Let $Z$ be a connected complex manifold of dimension at least 2 endowed with a holomorphic action of a connected complex Lie group $L$. In this section we introduce general tools for constructing automorphisms of $Z$ other than those directly arising from the $L$-action. 

We start with the following simple proposition, where for two complex spaces $Y_1, Y_2$ we denote by $\hbox{Hol}(Y_1,Y_2)$ the  set of all holomorphic maps from $Y_1$ into $Y_2$.

\begin{proposition}\label{general} Let $Y$ be a complex space and $\sigma:Z\rightarrow Y$ an $L$-invariant holomorphic map. For every $f\in \hbox{Hol}(Y,L)$ define $F_f$ to be the following self-map of $Z$:
\begin{equation}
F_f(z):=f(\sigma(z))z,\quad z\in Z.\label{autom}
\end{equation}
Then $F_f\in\hbox{Aut}_{\mathcal O}(Z)$.
\end{proposition}

\begin{proof} Clearly, for any $f$ the map $F_f$ is holomorphic. Furthermore, the $L$-invariance of $\sigma$ implies that the map
$$
z\rightarrow [f(\sigma(z))]^{-1}z
$$
is the inverse of $F_f$. The proof is complete.
\end{proof}
  
Let ${\mathfrak l}$ be the Lie algebra of $L$. In the situation of Proposition \ref{general}, for every map $\varphi\in\hbox{Hol}(Y,{\mathfrak l})$ we consider the real 1-parameter family of maps $f_{\tau}:=\exp(\tau\varphi)$ in $\hbox{Hol}(Y,L)$. This family gives rise to the corresponding 1-parameter family 
$\{F_{f_{\tau}}\}\subset \hbox{Aut}_{\mathcal O}(Z)$ as defined in (\ref{autom}). One immediately observes that $\{F_{f_{\tau}}\}$ is a 1-parameter subgroup of $\hbox{Aut}_{\mathcal O}(Z)$ and that the corresponding action map 
$$
\RR\times Z\ra Z,\quad (\tau,z)\mapsto F_{f_{\tau}}(z)
$$
is real-analytic. Therefore, $\{F_{f_{\tau}}\}$ gives rise to a complete holomorphic vector field $v_{\varphi}$ on $Z$, and we obtain a continuous linear map $\Phi^L: \varphi\mapsto v_{\varphi}$ from $\hbox{Hol}(Y,{\mathfrak l})$ into $V_{\mathcal O}^c(Z)$, where both vector spaces are considered with the compact-open topology.

We will now study the injectiviy properties of the map $\Phi^L$. For $a\in{\frak l}$ we denote by $\hat a$ the corresponding complete holomorphic vector field on $Z$ generated by the $L$-action, that is,
$$
\hat a(z):=\displaystyle\frac{d}{dt}\bigl(\exp(-ta)z\bigr)|_{t=0}, \quad z\in Z.
$$
Further, for $y\in Y$ let $F_y:=\sigma^{-1}(y)$ be the $\sigma$-fiber over $y$, and let ${\mathfrak l}_y$ be the following ideal in ${\mathfrak l}$
$$
{\mathfrak l}_y:=\{a\in{\mathfrak l}: \hbox{$\hat a\equiv 0$ on $F_y$}\}.
$$
In terms of these ideals the kernel $\hbox{ker}\,\Phi^L$ of the map $\Phi^L$ is described as follows:
\begin{equation}
\hbox{ker}\,\Phi^L=\left\{\varphi\in\hbox{Hol}(Y,{\mathfrak l}): \varphi(y)\in{\mathfrak l}_y,\,\,\hbox{for all $y$ in $Y$}\right\}.\label{kerform}
\end{equation}

Assume now that the map $\sigma$ is surjective. For $z\in Z$ let $\hbox{rank}_z\,\sigma:=\hbox{codim}_z\,F_{\sigma(z)}$ and $Z':=\{z\in Z:\hbox{rank}_z\,\sigma=s\}$, where $s:=\hbox{dim}\,Y$. By Satz 17 in \cite{R}, the set $Z'$ is analytic Zariski open in $Z'$ (a complement to a closed complex analytic subset of $Z'$), and $\sigma$ maps $Z'$ onto the open subset $Y':=\sigma(Z')$ of $Y$. For $z\in Z'$ define $d_z^L:=\hbox{dim}\,{\mathfrak l}_{\sigma(z)}$.

The following proposition is the main result of this section.

\begin{proposition}\label{semicont} Let $Z$, $L$, $Y$ and $\sigma$ be as in Proposition \ref{general}, and suppose additionally that $\sigma$ is surjective. Then the function $d_z^L$ is upper semicontinuous on $Z'$, i.e., for every integer $k\ge 0$ the set $D_k:=\{z\in Z': d_z^L\ge k\}$ is closed in $Z'$.
\end{proposition}
 
\begin{proof} Let $a_1,\dots,a_{\ell}\in{\mathfrak l}$ and let $\{y_n\}$ be a sequence in $Y'$ converging to $y_0\in Y'$. Suppose further that for every $n$ there exists a non-zero vector $p_n=(p^1_n,\dots,p^{\ell}_n)\in\CC^{\ell}$ for which the sum $\sum_{j=1}^{\ell} p^j_n\hat a_j$ vanishes on $F_{y_n}$. Choosing the vectors $p_n$ in the unit sphere, we can assume that the sequence $\{p_n\}$ converges to some non-zero vector $p_0=(p^1_0,\dots,p^{\ell}_0)$.

Fix $z_0\in F_{y_0}$ and choose a closed connected submanifold $\Sigma$ of dimension $s$ in a neighborhood ${\mathcal V}$ of $z_0$, such that $\Sigma\cap F_{y_0}=\{z_0\}$. Since the set where $\hbox{rank}_z\,\sigma|_{\Sigma}$ takes the maximal value (equal to $s$) is open in $\Sigma$, by shrinking ${\mathcal V}$ we can assume that the intersection $\Sigma\cap F_{\sigma(z)}$ is discrete for all $z\in {\mathcal V}$. By Satz 19 of \cite{R}, the image $\sigma(\Sigma)$ is an open subset of $Y'$. Therefore, $\Sigma$ intersects the fiber $F_{y_n}$ for all sufficiently large $n$.

Choosing $z_n\in\Sigma\cap F_{y_n}$ we see that the sequence $\{z_n\}$ converges to $z_0$, hence the vector field $\sum_{j=1}^{\ell} p^j_0\hat a_j$ vanishes at $z_0$. Since $z_0$ was an arbitrary point of the fiber $F_{y_0}$, it follows that this vector field vanishes on all of $F_{y_0}$. 

Let $m:=\hbox{dim}\,{\mathfrak l}$ and let $\{w_n\}$ be a sequence in the set $D_k$ converging to a point $w_0\in Z'$. Clearly, for any $b_1,\dots, b_{m-k+1}\in{\mathfrak l}$ and any $n$ there exists a vector $q_n=(q^1_n,\dots,q^{m-k+1}_n)\in\CC^{m-k+1}$ such that $\sum_{j=1}^{m-k+1}q^j_n\hat b_j$ vanishes on $F_{\sigma(w_n)}$. As we showed above, there exists a non-zero vector $q_0=(q^1_0,\dots,q^{m-k+1}_0)$ such that $\sum_{j=1}^{m-k+1}q^j_0\hat b_j$ vanishes on $F_{\sigma(w_0)}$, hence $w_0$ lies in $D_k$.

The proof of the proposition is complete. \end{proof}

\begin{corollary}\label{openness} Let $Z$, $L$, $Y$ and $\sigma$ be as in Proposition \ref{semicont}, and let $\mu^L:=\hbox{min}_{z\in Z'}\,d_z^L$. Then the set $D^L:=\{z\in Z': d_z^L=\mu^L\}$ is open in $Z'$.
\end{corollary}

In the following proposition we apply the upper semicontinuity of $d_z^L$ to obtain the infinite-dimensionality of $V^c_{\mathcal O}(Z)$. This proposition will be used to prove Theorem \ref{main} in the subsequent sections.
  
\begin{proposition}\label{mainprop} Let $Z$, $L$, $Y$ and $\sigma$ be as in Proposition \ref{semicont}. Assume, furthermore, that the action of $L$ on $Z$ is almost effective and that $Y$ admits a non-constant holomorphic function. Then $V^c_{\mathcal O}(Z)$ is infinite-dimensional. 
\end{proposition}

\begin{proof} Fix a complex 1-parameter subgroup $L_1$ in $L$ and consider the corresponding map $\Phi^{L_1}$. Since $L_1$ acts on $Z$ almost effectively, it follows that $\mu^{L_1}=0$. Therefore, by (\ref{kerform}) every function $\varphi\in\hbox{ker}\,\Phi^{L_1}$ vanishes on the set $D^{L_1}$. Since by Corollary \ref{openness} the set $D^{L_1}$ is open in $Z'$, the function $\varphi$ vanishes everywhere on $Y$, and we obtain that $\hbox{ker}\,\Phi^{L_1}$ is trivial.

Further, since $Y$ admits a non-constant holomorphic function, say $f$, the algebra ${\mathcal O}(Y)$ is infinite-dimensional (this can be seen by considering the compositions $h\circ f$, where $h\in{\mathcal O}(f(Y)))$. Thus, $\Phi^{L_1}$ is a continuous injective linear map from the infinite-dimensional vector space ${\mathcal O}(Y)$ into $V^c_{\mathcal O}(Z)$. Hence $V^c_{\mathcal O}(Z)$ is infinite-dimensional.
\end{proof}

\begin{remark}\label{arb} In the proof of Proposition \ref{mainprop} we observed that the map $\Phi^L$ is injective if $L$ is a 1-dimensional group acting effectively. In fact, under an additional assumption on $Y$, the following characterization of the size of the quotient $\hbox{Hol}(Y,{\mathfrak l})/\hbox{ker}\,\Phi^L$ holds for groups of an arbitrary dimension $m\ge 1$. Let $Z$, $L$, $Y$ and $\sigma$ be as in Proposition \ref{semicont}. Assume, furthermore, that the action of $L$ on $Z$ is almost effective and that there exist $f_1,\dots,f_s\in{\mathcal O}(Y)$ that give coordinates near a smooth point in $Y$. Then there exists a continuous linear surjective map $\Psi^L$ from $\hbox{Hol}(Y,{\mathfrak l})/\hbox{ker}\,\Phi^L$ onto a dense subset of ${\mathcal O}^{m-\mu^L}(V)$, where $V$ is a small neighborhood of a smooth point in $Y$ (note that $\mu^L<m$ since the action of $L$ is almost effective). If $L$ is 1-dimensional, then $\hbox{ker}\,\Phi^L$ is trivial and $\Psi^L$ is simply the restriction map ${\mathcal O}(Y)\ra{\mathcal O}(V)$. 
\end{remark}

We will apply Proposition \ref{mainprop} in situations of two kinds. In these situations the space $Y$ will be Stein. Suppose first that the manifold $Z=M/N$ is homogeneous under a holomorphic action of a connected complex Lie group $M$. If one can find a closed normal complex subgroup $N_1$ of $M$ containing $N$ then the quotient map $\Pi: M/N\ra M/N_1$ is $N_1$-invariant. If, furthermore, $M/N_1$ is a positive-dimensional Stein manifold, then $Y:=M/N_1$, $\sigma:=\Pi$ and $L:=N_1/K$, satisfy the assumptions of Proposition \ref{mainprop}, where $K$ is the ineffectivity subgroup of the $N_1$-action on $Z$. This line of argument will be used to prove Theorem \ref{main} for solvable groups $G$ in Section \ref{case2}.

Another situation of interest to us in which Proposition \ref{mainprop} will be applied comes from categorical quotients. Suppose that a connected complex Lie group $M$ acts on $Z$ holomorphically. For any complex subgroup $L$ of $M$ define $Z/\hspace{-0.15cm}/L$ to be the categorical quotient of $Z$ by $L$, that is, the quotient of $Z$ by the equivalence relation $\sim$, where $z_1\sim z_2$ if and only if $f(z_1)=f(z_2)$ for all $f$ in the algebra ${\mathcal O}(Z)^L$ of $L$-invariant holomorphic function on $Z$. We denote by $\pi: Z\ra Z/\hspace{-0.15cm}/L$ the quotient map and consider $Z/\hspace{-0.15cm}/L$ with the corresponding quotient topology. Assume now that $Z$ is a Stein manifold and $L$ a complex reductive group. It was shown in \cite{S} that under these assumptions the quotient $Z/\hspace{-0.15cm}/L$ can be given the structure of a Stein complex space in such a way that the map $\pi$ becomes holomorphic and $\pi^*:{\mathcal O}(Z/\hspace{-0.15cm}/L)\ra {\mathcal O}(Z)^L$ is an isomorphism. In the remainder of this section $Z$ is assumed to be a Stein manifold and $L$ a reductive complex Lie group.  

For our applications we will need the following four properties of reductive group actions on Stein manifolds and the corresponding categorical quotients. Proofs and further properties can be found in \cite{S} (see also \cite{HH}).

(P1) The $L$-action is of algebraic type, i.e., every $L$-orbit $L.z$ is open in its analytic Zariski closure in $Z$.

(P2) $z_1\sim z_2$ if and only if $\overline{L.z_1}\cap\overline{L.z_2}\ne\emptyset$. 

(P3) If $\psi: Z\ra Z_1$ is an $L$-invariant holomorphic map, then there exists a unique holomorphic map $\tau_{\psi}:Z/\hspace{-0.15cm}/L\ra Z_1$ such that $\psi=\tau_{\psi}\circ\pi$. If $\psi^*:{\mathcal O}(Z_1)\ra{\mathcal O}(Z)^L$ is an isomorphism, then $\tau_{\psi}$ is biholomorphic. 

(P4) For every $y\in Z/\hspace{-0.15cm}/L$ the fiber $\pi^{-1}(y)$ contains exactly one closed $L$-orbit (the unique orbit of minimal dimension).

To be able to apply Proposition \ref{mainprop} to $Y:=Z/\hspace{-0.15cm}/L$ and $\sigma:=\pi$, we only need to ensure that $Z/\hspace{-0.15cm}/L$ is positive-dimensional. Due to property (P4) a sufficient condition for the positive-dimensionality of $Z/\hspace{-0.15cm}/L$ is the existence of two closed $L$-orbits in $Z$. For $L\simeq\CC^*$ another sufficient condition is given by the following simple lemma.

\begin{lemma}\label{suffpositive} Assume that $L\simeq\CC^*$ and that $L$ has a 1-dimensional closed orbit in $Z$. Then $\hbox{dim}\,Z/\hspace{-0.15cm}/L>0$.
\end{lemma}

\begin{proof} If for $z_0\in Z$ the orbit $L.z_0$ is 1-dimensional and closed, then by property (P4) it is the only closed $L$-orbit contained in the fiber $\pi^{-1}(\pi(z_0))$. By property (P1) the $L$-action is of algebraic type, hence any other $L$-orbit $O$ lying in $\pi^{-1}(\pi(z_0))$ is open in its analytic Zariski closure in $Z$ and by property (P2) contains $L.z_0$ in its boundary. In particular, $O$ must have dimension greater than 1, which is impossible since $L$ is 1-dimensional. Therefore, the fiber $\pi^{-1}(\pi(z_0))$ coincides with $L.z_0$. Since $\hbox{dim}\,Z\ge 2$, there exists another closed $L$-orbit in $Z$, and the proof is complete.
\end{proof}   

We will apply Proposition \ref{mainprop} in the situation of Lemma \ref{suffpositive} to prove Theorem \ref{main} for non-solvable groups $G$ in Section \ref{case1}.

\begin{remark}\label{algebraic} If $Z=M/N$ is a homogeneous space of a holomorphic action of a complex connected reductive Lie group $M$, then the following conditions are equivalent: (i) $Z$ is Stein, (ii) $Z$ is affine algebraic, (iii) $N$ is reductive, and in the affine algebraic realization of $Z$ the action of $M$ is algebraic (see \cite{Mat2}, \cite{Mat3}, \cite{O}, \cite{BH-C}, \cite{B-B}). In this case for a subgroup $L\simeq\CC^*$ of $M$ the classical algebraic invariant theory supplies all tools required to prove Theorem \ref{main}, as outlined below. 

Let ${\mathcal P}(Z)^L$ be the algebra of $L$-invariant regular functions on $Z$. Define an equivalence relation $\sim^{\hbox{\tiny alg}}$ on $Z$ by setting $z_1\sim^{\hbox{\tiny alg}} z_2$ if and only if $f(z_1)=f(z_2)$ for every $f\in{\mathcal P}(Z)^L$. Let $Z/\hspace{-0.15cm}/^{\hbox{\tiny alg}}L$ be the quotient of $Z$ by $\sim^{\hbox{\tiny alg}}$ endowed with the quotient topology. By a result due to Hilbert, the algebra ${\mathcal P}(Z)^L$ is finitely-generated (see, e.g.,\cite{Kr}, p. 95). Using this fact, one can introduce the structure of an affine algebraic variety on $Z/\hspace{-0.15cm}/^{\hbox{\tiny alg}}L$ in such a way that the quotient map $\pi^{\hbox{\tiny alg}}:Z\ra Z/\hspace{-0.15cm}/^{\hbox{\tiny alg}}L$ becomes an algebraic morphism, and $(\pi^{\hbox{\tiny alg}})^*:{\mathcal P}(Z/\hspace{-0.15cm}/^{\hbox{\tiny alg}}L)\ra {\mathcal P}(Z)^L$ is an isomorphism, where ${\mathcal P}(Z/\hspace{-0.15cm}/^{\hbox{\tiny alg}}L)$ is the algebra of regular functions on $Z/\hspace{-0.15cm}/^{\hbox{\tiny alg}}L$. Furthermore, algebraic analogues of properties (P1)--(P4) hold.

Comparing (P2) with its algebraic counterpart, we see that $Z/\hspace{-0.15cm}/L$ and $Z/\hspace{-0.15cm}/^{\hbox{\tiny alg}}L$ coincide as topological spaces and $\pi=\pi^{\hbox{\tiny alg}}$. Furthermore, $\pi^*: {\mathcal O}(Z/\hspace{-0.15cm}/^{\hbox{\tiny alg}}L)\ra{\mathcal O}(Z)^L$ is an isomorphism, hence by (P3) the complex spaces $Z/\hspace{-0.15cm}/L$ and $Z/\hspace{-0.15cm}/^{\hbox{\tiny alg}}L$ are biholomorphic. For more details on algebraic quotients see \cite{Kr}, \cite{MFK}, \cite{HH}.
\end{remark}

\section{Proof of Theorem \ref{main} for $G$ non-solvable}\label{case1}
\setcounter{equation}{0}

Without loss of generality one can assume that $G$ acts on $X$ effectively. We write the Levi-Malcev decomposition for $G$ as a locally semi-direct product $G=R\cdot S$, where $R$ is the radical and $S$ the semi-simple Levi subgroup of $G$. In this section we assume that $S$ is non-trivial. 

Let $A$ be a subgroup of $S$ isomorphic to $SL_2(\CC)$ (such a subgroup always exists). Since $A$ is a complex reductive group, the action of $A$ on $X$ is of algebraic type (see property (P1)). Let $T$ be a maximal torus in $A$. Since $T$ is a complex reductive subgroup, its action on $X$ is of algebraic type as well.

The result of this section is the following lemma.

\begin{lemma}\label{closedorbit} There exists a 1-dimensional closed $T$-orbit in $X$.
\end{lemma}

For $S$ non-trivial, Theorem \ref{main} immediately follows from Lemma \ref{closedorbit}, and from Proposition \ref{mainprop} and Lemma \ref{suffpositive} applied to $Z:=X$, $Y:=X/\hspace{-0.15cm}/T$, $L:=T$, $\sigma:=\pi$, where $\pi$ is the categorical quotient map $X\ra X/\hspace{-0.15cm}/T$. In the case when $G$ is reductive (e.g., semi-simple), Theorem \ref{main} follows from the classical algebraic invariant theory by setting $Y:=X/\hspace{-0.15cm}/^{\hbox{\tiny alg}}T$ and $\sigma:=\pi^{\hbox{\tiny alg}}$, where $\pi^{\hbox{\tiny alg}}$ is the algebraic categorical quotient map $X\ra X/\hspace{-0.15cm}/^{\hbox{\tiny alg}}T$ (see Remark \ref{algebraic}). 

{\it Proof of Lemma \ref{closedorbit}.}\,\,\,----\,\,\, Choose a point $q\in X$ that is not fixed by $A$, and let $B$ be the isotropy subgroup of $q$ with respect to the $A$-action. Then either (a) $\hbox{dim}\, A/B=1$, or (b) $\hbox{dim}\, A/B=2$, or (c) $B$ is finite. In Case (a) the quotient $A/B$ is equivalent to $\PP^1$, which is impossible since $X$ is Stein. 

In Case (b) the quotient $A/B$ is equivalent to either the affine quadric $Q_{(2)}$, or  $\PP^2\setminus Q_1$ (note that $Q_{(2)}$ is a two-sheeted cover of $\PP^2\setminus Q_1$), or $\CC^2\setminus\{0\}/\ZZ_m$, $m\ge 1$ (see \cite{HL}). In this case most $T$-orbits in $A/B$ are closed. Hence if the orbit $A.q$ is closed in $X$, there exists a closed 1-dimensional $T$-orbit in $X$. If $A.q$ is not closed, then we consider the analytic Zariski closed set $S_1=\overline{A.q}\setminus A.q$ in $X$. The set $S_1$ is $A$-invariant and by properties (P2), (P4) contains a single $A$-orbit of minimal dimension. If $S_1$ is 1-dimensional, it follows that it contains a 1-dimensional $A$-orbit, and we obtain a contradiction as in Case (a). Otherwise, $S_1$ is a single point. In this case $\overline{A.q}$ is a disjoint union of this point and a copy of $\CC^2\setminus\{0\}/\ZZ_m$, hence the generic $T$-orbit in $A.q$ is closed in $X$. 

Finally, in Case (c) all $T$-orbits in $A/B$ are closed. Hence if the orbit $A.q$ is closed in $X$, there exists a closed 1-dimensional $T$-orbit in $X$. If $A.q$ is not closed, then we consider the analytic Zariski closed set $S_2=\overline{A.q}\setminus A.q$ in $X$. This set contains a 2-dimensional $A$-orbit, and replacing $q$ by a point from this orbit, we apply the argument used in Case (b). 

Thus, we have shown that there exists a 1-dimensional closed $T$-orbit in $X$.            
\qed

\section{Proof of Theorem \ref{main} for $G$ solvable}\label{case2}
\setcounter{equation}{0}

In this section we assume that $G=R$ is solvable. By allowing $G$ to act on $X$ almost effectively, we further assume that $G$ is simply-connected. The following is the key result for the proof of Theorem \ref{main} in this case.

\begin{proposition}\label{solv} Let $X$ with $\hbox{dim}\,X\ge 2$ be a connected Stein manifold. Assume that $X$ is a homogeneous space $G/H$ of a simply-connected solvable complex Lie group $G$ acting almost effectively on $X$. Then there exists a closed subgroup $I\subset G$ containing $H$ such that $\hbox{dim}\,I>\hbox{dim}\, H$ and $G/I$ is positive-dimensional Stein manifold. Furthermore, either $G/I$ is biholomorphic to $\CC^m$ for $m\ge 1$, or $G'\subset I$.  
\end{proposition}

We will now show how Theorem \ref{main} for $G$ solvable follows from Proposition \ref{solv}. We proceed by induction on $n:=\hbox{dim}\,X$. Let $n=2$ and $I$ be any subgroup as in Proposition \ref{solv}. If $G/I$ is biholomorphic to $\CC$, then by Grauert's Oka Principle \cite{G} the fiber bundle $\Pi: G/H\ra G/I$ is holomorphically trivial, since it is topologically trivial. Hence the fiber $I/H$ of this bundle is connected (note also that $I$ is connected since $G/I$ is simply-connected). Since $I/H$ is 1-dimensional and the complex Lie group $I$ acts holomorphically and transitively on $I/H$, it follows that $I/H$ is biholomorphic to either $\CC$ or $\CC^*$. Hence the manifold $X$ is biholomorphic to either $\CC^2$ or $\CC\times\CC^*$, and therefore $V_{\mathcal O}^c(X)$ is infinite-dimensional. 

Assume now that $G'\subset I$, hence $I$ is a normal subgroup of $G$. The infinite-dimensionality of $V_{\mathcal O}^c(X)$ now follows from Proposition \ref{mainprop} applied to $Z:=X$, $Y:=G/I$, $L:=I$, $\sigma:=\Pi$, where $\Pi: G/H\ra G/I$ is the factorization map.

We will now deal with the induction step for $n\ge 3$. If $G/I$ is biholomorphic to $\CC^m$, then using Grauert's Oka Principle we see as above that $X$ is biholomorphic to $I/H\times\CC^m$ (here $I/H$ and $I$ are connected). If $m\ge 2$, then $V_{\mathcal O}^c(X)$ is infinite-dimensional since $V_{\mathcal O}^c(\CC^m)$ is infinite-dimensional. If $m=1$, then $\hbox{dim}\,I/H=n-1\ge 2$, and we apply the induction hypothesis to the Stein manifold $I/H$ (note that $I/H$ is a homogeneous space of a simply-connected solvable complex Lie group acting almost effectively). Then $V_{\mathcal O}^c(X)$ is infinite-dimensional since $V_{\mathcal O}^c(I/H)$ is infinite-dimensional. If $G'\subset I$, the infinite-dimensionality of $V_{\mathcal O}^c(X)$ follows from Proposition \ref{mainprop} as in the preceding paragraph.

In the remainder of this section we prove Proposition \ref{solv}.

{\it Proof of Proposition \ref{solv}.}\,\,\,----\,\,\, Let $J:=N_G(H^{\circ})$ be the normalizer of the connected identity component $H^{\circ}$ of $H$ in $G$. The subgroup $J$ is closed in $G$ and clearly contains $H$. The corresponding fibration $G/H\ra G/J$ is often called the Tits fibration (see \cite{BR}, \cite{Ti}). This fibration coincides with the ${\mathfrak g}$-anticanonical fibration, where ${\mathfrak g}$ is the Lie algebra of $G$ (see \cite{HO1}, p.p. 61, 65). In particular, there is a representation $\rho$ of $G$ (not necessarily faithful) by transformations of a complex projective space $\PP^r$ such that $J=\rho^{-1}({\mathcal J})$, where ${\mathcal J}$ consists of all elements of $\rho(G)$ that fix some point $x_0$ in $\PP^r$. By a result due to Chevalley \cite{C} (see also \cite{HO1}, p. 31), the commutator subgroup $\rho(G)'=\rho(G')$ is an algebraic subgroup of $\hbox{Aut}_{\mathcal O}(\PP^r)$. Hence the orbit $\rho(G)'.x_0$ is closed in $\rho(G).x_0$, and thus the subgroup $P:=\rho^{-1}(\rho(G)')J$ is closed in $G$. We will consider three main cases. 

{\bf Case 1.} Assume that $J\ne G$ and $P=G$. In this case $\rho(G).x_0=\rho(G)'.x_0$. Let $U$ be the isotropy subgroup of $x_0$ under the $\rho(G)'$-action. Since $\rho(G)'$ acts on $\PP^r$ algebraically, $U$ has finitely many connected components. The group $\rho(G)'$ consists of unipotent matrices, hence it is simply-connected and nilpotent, and therefore its exponential map is a biholomorphism onto $\rho(G)'$. It then follows that $U$ is in fact connected (observe that a subgroup with finitely many connected components in a Lie group for which the exponential map is a diffeomorphism, is connected). Hence $\rho(G)'/U$ is biholomorphic to some $\CC^{\ell}$ (see, e.g., \cite{OR}). Thus $G/J$ is biholomorphic to $\CC^{\ell}$ as well. Since $J\ne G$, we have $\ell\ge 1$. If $\hbox{dim}\,J>\hbox{dim}\,H$, then taking $I:=J$, we have the result of the proposition.

Let $\hbox{dim}\,J=\hbox{dim}\,H$. From the simple connectivity of $G/J$ we then see that $J=H=H^{\circ}$. Let $Q:=\rho^{-1}(\rho(G)')$ and denote by ${\mathfrak q}$ the Lie algebra of $Q$. We now consider the adjoint representation $\hbox{Ad}_G$ of $G$ on ${\mathfrak q}$ (note that ${\mathfrak q}$ is an ideal in ${\mathfrak g}$ and hence is $\hbox{Ad}_G$-invariant). By a theorem due to Lie (see, e.g., \cite{OV}, p. 8), there exists a full $\hbox{Ad}_G$-invariant flag $\{0\}={\mathfrak q}_0\subset{\mathfrak q}_1\subset\dots\subset{\mathfrak g}_M={\mathfrak q}$ in ${\mathfrak q}$ (here $M=\hbox{dim}_{\CC}{\mathfrak q}$). Since each subspace in the flag is $\hbox{ad}_G$-invariant, it is in fact an ideal in ${\mathfrak g}$. 

Let $G_k$ be the normal connected subgroup of $G$ with Lie algebra ${\mathfrak q}_k$, $k=0,\dots,M$  ($G_k$ is automatically closed since any connected subgroup of a simply-connected solvable group is closed -- see, e.g., \cite{Vara}, p. 243). For every $k$, the image $\rho(G_k)$ is a connected subgroup of the simply-connected nilpotent group $\rho(G)'$ and hence is an algebraic subgroup in $\rho(G)'$ (note that every simply-connected nilpotent group is algebraic and every connected subgroup in it is an algebraic subgroup -- see, e.g., \cite{GSV}, p. 43). It follows as before that the orbit $\rho(G_k).x_0$ is closed in $\rho(G).x_0$, and thus the subgroup $P_k:=\rho^{-1}(\rho(G_k))J$ is closed in $G$ for every $k$. Let $1\le k_0\le M-1$ be such that $\hbox{dim}\,J<\hbox{dim}\,P_{k_0}<\hbox{dim}\,P=\hbox{dim}\,G$. Since $J=H$ is connected, it lies in $P_{k_0}^{\circ}$. The quotient $G/P_{k_0}^{\circ}$ is biholomorphic to some $\CC^k$ (see, e.g., \cite{OR}), and taking $I:=P_{k_0}^{\circ}$ we obtain the result of the proposition.

{\bf Case 2.} Assume that $P\ne G$. Let $\hat G$ be the Zariski closure of $\rho(G)$ in $\hbox{Aut}_{\mathcal O}(\PP^r)$ (note that $\hat G$ is automatically a subgroup of $\hbox{Aut}_{\mathcal O}(\PP^r)$). Consider the orbit $\hat G.x_0$ and let $\hat J$ be the isotropy subgroup of $x_0$ under the $\hat G$-action. Since $\hat G'=\rho(G)'$, we see as before that the orbit $\rho(G)'.x_0$ is closed in $\hat G.x_0$. Therefore, the normal subgroup $\rho(G)'\hat J$ is closed in $\hat G$. The quotient $\hat G/\rho(G)'\hat J$ is an affine algebraic Abelian group, and, in particular, is Stein. The complex Abelian group $G/P$ can be realized in $\hat G/\rho(G)'\hat J$ as the $\rho(G)$-orbit of the neutral element $\rho(G)'\hat J$ and therefore is holomorphically separable. It now follows from \cite{M} that $G/P$ is isomorphic to a product $\CC^p\times(\CC^*)^q$ and thus is Stein. 

Since $P\ne G$, the quotient $G/P$ is positive-dimensional. We will now show that $\hbox{dim}\,P>\hbox{dim}\,H$. If this is not the case, then $J^{\circ}=H^{\circ}$ and $(\rho^{-1}(\rho(G)'))^{\circ}\subset H^{\circ}$. In particular, the connected normal subgroup $G'$ lies in $H$. Therefore, $G'$ acts on $X$ trivially. Since the action of $G$ on $X$ was assumed to be almost effective, it follows that $G'$ is trivial, i.e., $G$ is Abelian. But in this case $J=G$ which contradicts the assumption $P\ne G$. Thus we have $\hbox{dim}\,P>\hbox{dim}\,H$, and taking $I:=P$, we obtain the result of the proposition.

{\bf Case 3.} Assume that $J=G$. In this case $H^{\circ}$ is normal in $G$. Since $G$ acts on $X$ almost effectively, the subgroup $H$ is discrete. To emphasize the discreteness of $H$, we redenote it by $\Gamma$. We now consider several possibilities.

{\bf 3.1} Suppose that $G$ is nilpotent. By the results of \cite{Mal}, \cite{Mat1} there exists a closed connected real subgroup $\Gamma_{\RR}$ of $G$ (called the real hull of $\Gamma$) that contains $\Gamma$ and such that $X_{\RR}:=\Gamma_{\RR}/\Gamma$ is compact. The Lie algebra $\gamma_{\RR}$ of $\Gamma_{\RR}$ is the smallest real subalgebra of ${\mathfrak g}$ containing the subset $\log(\Gamma)\subset{\mathfrak g}$. Let $\gamma_{\CC}$ be the smallest complex subalgebra of ${\mathfrak g}$ containing $\gamma_{\RR}$. Denote by $\Gamma_{\CC}$ the connected closed complex Lie subgroup of $G$ with Lie algebra $\gamma_{\CC}$. The quotient $G/\Gamma_{\CC}$ is biholomorphic to some $\CC^s$. 

{\bf 3.1.1.} Let $\Gamma_{\CC}\ne G$.  If $\Gamma_{\CC}$ is non-trivial, then taking $I:=\Gamma_{\CC}$, we have the result of the proposition. If $\Gamma_{\CC}$ is trivial, i.e., $\Gamma$ is trivial, let $Z$ be the center of $G$. Since $G$ is nilpotent, $Z$ is a positive-dimensional closed connected subgroup of $G$ (the center of any solvable exponential Lie group is connected \cite{Wu}). If $G$ is not Abelian, then $Z\ne G$, and $G/Z$ is biholomorphic to some $\CC^\tau$, $\tau\ge 1$. Then, taking $I:=Z$, we have the result of the proposition. If $G$ is Abelian, it is isomorphic to $\CC^{s}$ with $s\ge 2$. In this case we take $I:=\CC$ and obtain the result of the proposition.     

{\bf 3.1.2.} Let $\Gamma_{\CC}=G$. We consider two situations. 

{\bf 3.1.2.a.} Assume that $G$ is not Abelian. As before, let $Z$ be the center of $G$. The subgroup $Z\Gamma$ is closed in $G$ (see \cite{GH}), and we consider the fibration $G/\Gamma\ra G/Z\Gamma$. We represent $G/Z\Gamma$ as $G^1/\Gamma^1$, where $G^1:=G/Z$, $\Gamma^1:=\Gamma/\Gamma\cap Z$. Here $G^1$ is a simply-connected complex nilpotent subgroup and $\Gamma^1$ is a discrete subgroup such that $\Gamma^1_{\CC}=G^1$. Applying the same procedure to $G^1$, $\Gamma^1$ and the center $Z^1$ of $G^1$, we obtain the fibration $G^1/\Gamma^1\ra G^1/Z^1\Gamma^1=G^2/\Gamma^2$. Repeating this process finitely many times, we obtain for some $N\ge 1$ a simply-connected complex nilpotent group $G^N$ and a discrete subgroup $\Gamma^N$ with positive-dimensional quotient $G^N/\Gamma^N$, such that $G^N/Z^N\Gamma^N$ is trivial, where $Z^N$ is the center of $G^N$. Hence $G^N$ is Abelian, and therefore $(G^{N-1})'\subset Z^{N-1}$ (we set $G^0:=G$, $Z^0:=Z$, $\Gamma^0:=\Gamma$). Taking the inverse image of $Z^{N-1}$ under the chain of factorization maps $G\ra G/Z=G^1\ra G^1/Z^1=G^2\ra\dots\ra G^{N-2}/Z^{N-2}=G^{N-1}$, we obtain a subgroup $V\subset G$ containing $G'$. Clearly, $V\Gamma$ is the inverse image of $Z^{N-1}\Gamma^{N-1}$ and hence is closed in $G$. Furthermore, $G/V\Gamma=G^{N-1}/Z^{N-1}\Gamma^{N-1}=G^N/\Gamma^N$.

We note that $\hbox{dim}\,V\Gamma>0$ since otherwise $V^{\circ}$ is trivial, and therefore the connected normal subgroup $G'$ is trivial as well, which contradicts our assumption. 

\begin{lemma}\label{lemmasteinness} The manifold $G^N/\Gamma^N$ is Stein.
\end{lemma}

\begin{proof} Arguing as in the proof of Theorem 7 of \cite{GH}, using the holomorphic separability of $X$, we see that $X_{\RR}$ is totally real in $X$ and $\hbox{dim}\,X_{\RR}=\hbox{dim}_{\CC}X$. This implies that $X^1_{\RR}:=\Gamma^1_{\RR}/\Gamma^1$ is totally real in $X^1:=G^1/\Gamma^1$, where $\Gamma^1_{\RR}$ is the real hull of $\Gamma^1$, and $\hbox{dim}\,X^1_{\RR}=\hbox{dim}_{\CC}X^1$. 

Descending along the chain of fibrations constructed above, we obtain that $X^N_{\RR}:=\Gamma^N_{\RR}/\Gamma^N$ is totally real in $X^N:=G^N/\Gamma^N$, where $\Gamma^N_{\RR}$ is the real hull of $\Gamma^N$, and $\hbox{dim}\,X^N_{\RR}=\hbox{dim}_{\CC}X^N$. Since $X^N$ is complex Abelian group and $X^N_{\RR}\subset X^N$ is compact, $X^N_{\RR}$ is a (real) Abelian group isomorphic to some torus $(S^1)^{\nu}$. Therefore, $X^N$ is isomorphic to $(\CC^*)^{\nu}$ and hence is Stein.
\end{proof} 

Taking $I:=V\Gamma$, we obtain the result of the proposition.

{\bf 3.1.2.b.} Assume that $G$ is Abelian. Since $X_{\RR}\subset X$ is compact, $X_{\RR}$ is a real Abelian group isomorphic to some torus $(S^1)^{\kappa}$, hence the complex Abelian group $X$ is isomorphic to $(\CC^*)^{\kappa}$ with $\kappa\ge 2$. Let $L\subset X$ be a subgroup isomorphic to $\CC^*$. Taking $I$ to be the inverse image of $L$ under the factorization map $G\ra G/\Gamma=X$, we have the result of the proposition.

{\bf 3.2.} Suppose that $G$ is not nilpotent. 

{\bf 3.2.1.} Let $\Gamma$ be trivial. Note that $G'\ne G$ (otherwise $G$ is not solvable) and that $G'$ is positive-dimensional (otherwise $G$ is Abelian and hence nilpotent). Since $G'$ is connected, $G/G'$ is isomorphic to some $\CC^{\alpha}$. Taking $I:=G'$, we obtain the result of the proposition.

{\bf 3.2.2.} Let $\Gamma$ be non-trivial. By \cite{HO2} there is a nilpotent subgroup $\Gamma_1\subset \Gamma$ of finite index. Let $\hat G$ be the algebraic hull of $G$ as constructed in \cite{HM}. Namely, $\hat G$ is a solvable complex algebraic group isomorphic as a Lie group to $(\CC^*)^{\beta}\ltimes G$ for some $\beta$. It contains $G$ as a closed subgroup in the real topology and as a dense subgroup in the Zariski topology. Let $\hat \Gamma_1$ be the Zariski closure of $\Gamma_1$ in $\hat G$. Clearly, $\hat \Gamma_1$ is a nilpotent algebraic subgroup of $\hat G$. Since $\Gamma_1$ is of finite index in $\Gamma$, it follows that $\Gamma\hat \Gamma_1$ is Zariski closed in $\hat G$. Hence the Zariski closure $\hat \Gamma$ of $\Gamma$ in $\hat G$ coincides with $\Gamma\hat \Gamma_1$ (it follows that $\Gamma\hat \Gamma_1$ is a subgroup of $\hat G$). Since $\hat \Gamma_1$ is of finite index in $\hat \Gamma$, we have $(\hat \Gamma)^{\circ}=(\hat \Gamma_1)^{\circ}$. In particular, $\hat \Gamma^{\circ}$ is nilpotent. 

Since $\hat G$ is an (affine) algebraic group, it is isomorphic to a linear algebraic group. In this representation the subgroup $G'=(\hat G)'$ consists of unipotent matrices, and hence is simply-connected and nilpotent. Moreover, by \cite{C} it is an algebraic subgroup in $\hat G$. In particular, the $G'$-orbit of the neutral element $\hat \Gamma$ in the algebraic variety $\hat G/\hat \Gamma$ is closed and, it follows that $G'\hat \Gamma$ is a closed subgroup of $\hat G$. Hence $G\cap G'\hat \Gamma$ is a closed subgroup of $G$.

{\bf 3.2.2.a.} Assume that $G\not\subset G'\hat \Gamma$, i.e., $G/(G\cap G'\hat \Gamma)$ is positive-dimensional. Observe that $G\cap G'\hat \Gamma$ is positive-dimensional as well (otherwise $G'$ would be trivial hence $G$ would be nilpotent). To see that the complex Abelian group $G/(G\cap G'\hat \Gamma)$ is Stein, we realize it in the Stein complex Abelian group $\hat G/G'\hat \Gamma$ as the $G$-orbit of the neutral element $G'\hat \Gamma$. Thus $G/(G\cap G'\hat \Gamma)$ is holomorphically separable, hence it is Stein (see \cite{M}). Therefore, taking $I:=G\cap G'\hat \Gamma$, we obtain the result of the proposition.

{\bf 3.2.2.b.} Assume that $G\subset G'\hat \Gamma$. In this case we consider the quotient $G/(G\cap \hat \Gamma)$. The subgroup $G'$ acts transitively on this quotient. If $G\subset \hat \Gamma$, then $G\subset (\hat \Gamma)^{\circ}$, which is impossible since $G$ is not nilpotent. Therefore $G/(G\cap \hat \Gamma)$ is positive-dimensional. Observe that $G/(G\cap \hat \Gamma)=G'/(G'\cap\hat \Gamma)$ can be realized in $\hat G/\hat \Gamma$ as the $G'$-orbit of the neutral element $\hat \Gamma$. The action of $G'$ on $\hat G/\hat \Gamma$ is algebraic, and therefore the subgroup $G'\cap\hat \Gamma$, being the isotropy subgroup of the neutral element $\hat \Gamma$, has finitely many connected components. Since $G'$ is simply-connected and nilpotent, it follows that $G'\cap\hat \Gamma$ is in fact connected. Therefore, $G/(G\cap \hat \Gamma)$ is biholomorphic to some $\CC^{\gamma}$. If $G\cap \hat \Gamma$ is zero-dimensional, then $\Gamma$ is trivial (note that $G/(G\cap \hat \Gamma)$ is simply-connected), which contradicts our assumption. Thus, taking $I:=G\cap \hat \Gamma$, we obtain the result of the proposition.

The proof of Proposition \ref{solv} is complete.\qed

Fakult\"at f\"ur Mathematik, Ruhr-Universit\"at Bochum, Universit\"atsstra\ss e 150, D-44801 Bochum, Germany \hspace{0.1cm}$\bullet$\hspace{0.1cm} {\tt e-mail: ahuck@cplx.ruhr-uni-bochum.de}
\vspace{1cm}

Department of Mathematics, The Australian National University, Canberra, ACT 0200, Australia \hspace{0.1cm}$\bullet$\hspace{0.1cm} {\tt e-mail: alexander.isaev@maths.anu.edu.au}

\end {document}